\documentclass [12pt] {article}
\usepackage{amsmath}
\usepackage{amssymb}
\usepackage{amsfonts}

\newtheorem{theorem}{Theorem}

\newtheorem{lemma}[theorem]{Lemma}

\newenvironment{proof}[1][Proof]{\noindent\textbf{#1.} }{\ \rule{0.5em}{0.5em}}
\begin{document}

\title{Explicit Euclidean Embeddings in Permutation Invariant Normed Spaces}
\author{Daniel J. Fresen, Yale Math}
\maketitle

\begin{abstract}
Let $(X,\left\Vert \cdot \right\Vert )$ be a real normed space of dimension $%
N\in \mathbb{N}$ with a basis $(e_{i})_{1}^{N}$ such that the norm is
invariant under coordinate permutations. Assume for simplicity that the
basis constant is at most $2$. Consider any $n\in \mathbb{N}$ and $%
0<\varepsilon <1/4$ such that $n\leq c(\log \varepsilon ^{-1})^{-1}\log N$.
We provide an explicit construction of a matrix that generates a $%
(1+\varepsilon )$ embedding of $\ell _{2}^{n}$ into $X$.
\end{abstract}

\section{Introduction}

The modern formulation of Dvoretzky's theorem states that there exists a
function $\xi (\cdot ):(0,1/2)\rightarrow (0,\infty )$ such that for all $%
(N,n,\varepsilon )\in \mathbb{N}\times \mathbb{N}\times (0,1/2)$ with $n\leq
\xi (\varepsilon )\log N$, and any real normed space $(X,\left\Vert \cdot
\right\Vert )$ with $\dim (X)=N$, there exists a linear map $T:\mathbb{R}%
^{n}\rightarrow X$ such that for all $x\in \mathbb{R}^{n}$,%
\begin{equation}
(1-\varepsilon )|x|\leq \left\Vert Tx\right\Vert \leq (1+\varepsilon )|x|
\label{1+epsilon embedding}
\end{equation}%
where $\left\vert \cdot \right\vert $ denotes the standard Euclidean norm.
The normed space $(\mathbb{R}^{n},\left\vert \cdot \right\vert )$ is denoted 
$\ell _{2}^{n}$. Equation (\ref{1+epsilon embedding}) expresses the fact
that $\ell _{2}^{n}$ can be $(1+\varepsilon )$ embedded in $X$.
Geometrically, this means that any centrally symmetric convex body in high
dimensional Euclidean space has cross-sections of lower dimension that are
approximately ellipsoidal. The logarithmic dependence on $N$, which is due
to Milman \cite{Mil}, is optimal in the general setting (specifically for $%
X=\ell _{\infty }^{N}$) but can be greatly improved for other spaces such as 
$\ell _{1}^{N}$, where one can take $n=\left\lfloor c\varepsilon
^{2}N\right\rfloor $ for some universal constant $c>0$. On the other hand,
the optimal dependence on $\varepsilon $ is unknown. The best current bound
is $\xi (\varepsilon )=c\varepsilon (\log \varepsilon ^{-1})^{-2}$ by
Schechtman \cite{Sch1}. If weaker forms of Knaster's problem are true \cite%
{Mil88, KaSz}, then one could take $\xi (\varepsilon )=c(\log \varepsilon
^{-1})^{-1}$, which would be optimal since this is the correct dependence in 
$\ell _{\infty }^{N}$. We refer to \cite{MiSc, Sch1, Sch2} for a more
detailed background.

The proofs of Dvoretzky's theorem typically make use of random embeddings
and are nonconstructive. A natural question (see for example Section 2.2 in 
\cite{JoSc} and Section 4 in \cite{SzICM}) is whether or not one can make
these and other probabilistic constructions in functional analysis explicit,
or at least decrease the randomness in some way.

Of course it is impossible to find an explicit Euclidean subspace of a
completely general and unspecified normed space, otherwise all subspaces
would be Euclidean. This follows from rotational invariance of the class of
symmetric convex bodies in $\mathbb{R}^{N}$ and the fact that the orthogonal
group $O(N)$ acts transitively on the Grassmannian $G_{N,n}$. The same is
true even if we assume that the space has nontrivial cotype, or that the
unit ball is in John's position. One either has to construct explicit
subspaces for specific spaces individually, or (most likely) we need to
impose some sort of symmetry in order to get a grip on the space.

The case $X=\ell _{1}^{N}$ is particularly important from the point of view
of applications, see for example \cite{IndSz} and the references therein.
For this space, there are various algorithms to compute embeddings or
decrease randomness \cite{AM06, AM07, GLR, GLW08, Ind00, Ind07, IndSz, LS,
Sch0}, although there is still no truly explicit embedding that is as good
as a random one. For $X=\ell _{p}^{N}$, where $p\in \mathbb{N}$ is an even
integer and%
\begin{equation*}
N\geq \binom{n+p-1}{p}
\end{equation*}%
which is true whenever $n\leq N^{1/p}$, the space $\ell _{2}^{n}$ embeds
isometrically into $X$, \cite{Mil88, KoKo}. In this case there are also
various explicit embeddings. For example, the identity%
\begin{equation*}
6\left( \sum_{i=1}^{4}x_{i}^{2}\right) ^{2}=\sum_{1\leq i<j\leq 4}\left(
x_{i}+x_{j}\right) ^{4}+\sum_{1\leq i<j\leq 4}\left( x_{i}-x_{j}\right) ^{4}
\end{equation*}%
defines an isometric embedding of $\ell _{2}^{4}$ into $\ell _{4}^{12}$. As
far as we are aware, there are no known explicit embeddings (in the
classical sense) or even algorithms, that apply to a wide class of spaces
such as\ those with an unconditional or permutation invariant basis (which
would be the two most natural cases to consider).

The purpose of this paper is to provide explicit $(1+\varepsilon )$
embeddings, as in (\ref{1+epsilon embedding}), of $\ell _{2}^{n}$ into a
general normed space $(X,\left\Vert \cdot \right\Vert )$ with a permutation
invariant basis, provided $n\leq c(\log \varepsilon ^{-1})^{-1}\log N$ and
we have some control over the basis constant. Our main result, Theorem \ref%
{Main result}, provides more details and complements (non-explicit) results
of Bourgain and Lindenstrauss \cite{BoLiMi} and Tikhomirov \cite{Tich}, who
studied symmetric spaces (spaces with a permutation invariant and $1$%
-unconditional basis). The dependence on both $N$ and $\varepsilon $ that we
provide is optimal.

\section{Main result}

For any $N\in \mathbb{N}$ let $S_{N}$ denote the permutation group on $N$
elements. Let $(X,\left\Vert \cdot \right\Vert )$ denote a real normed space
of dimension $N$ with a basis $(e_{i})_{1}^{N}$ such that the norm $%
\left\Vert \cdot \right\Vert $ is invariant under the action of $S_{N}$,
i.e. for all $(a_{i})_{1}^{N}\in \mathbb{R}^{N}$ and all $\sigma \in S_{N}$,%
\begin{equation*}
\left\Vert \sum_{i=1}^{N}a_{\sigma (i)}e_{i}\right\Vert =\left\Vert
\sum_{i=1}^{N}a_{i}e_{i}\right\Vert 
\end{equation*}%
The basis constant of $(e_{i})_{1}^{N}$ is defined as the smallest value of $%
K\geq 1$ such that for all $(a_{i})_{1}^{N}\in \mathbb{R}^{N}$ and all $%
j\leq N$,%
\begin{equation*}
\left\Vert \sum_{i=1}^{j}a_{i}e_{i}\right\Vert \leq K\left\Vert
\sum_{i=1}^{N}a_{i}e_{i}\right\Vert 
\end{equation*}%
Consider any $n\in \mathbb{N}$, $n\geq 6$, and $0<\varepsilon <(2K)^{-1}$.
Set $\delta =\varepsilon /1429$, $\sigma =\delta ^{-4}$ and $\alpha =2\delta
^{-4}(\log \delta ^{-1})^{1/2}$. Let $\Phi $ denote the standard normal
cumulative distribution, and for each integer point $x\in \mathbb{Z}^{n}\cap
\left( \alpha \sqrt{n}B_{2}^{n}\right) $ define%
\begin{equation}
m(x)=\left\lfloor N\prod_{i=1}^{n}\left( \Phi \left( \frac{x_{i}+1/2}{\sigma 
}\right) -\Phi \left( \frac{x_{i}-1/2}{\sigma }\right) \right) \right\rfloor 
\label{first multiplicity}
\end{equation}%
\begin{equation*}
N^{\prime }=\sum_{\substack{ x\in \mathbb{Z}^{n} \\ |x|\leq \alpha \sqrt{n}}}%
m(x)
\end{equation*}%
\begin{equation*}
m^{\prime }(x)=\left\{ 
\begin{array}{ccc}
m(x) & : & x\neq 0 \\ 
m(x)+N-N^{\prime } & : & x=0%
\end{array}%
\right. 
\end{equation*}%
Consider the $N\times n$ matrix $T$ defined as follows: For each $x\in 
\mathbb{Z}^{n}\cap \left( \alpha \sqrt{n}B_{2}^{n}\right) $, repeat the
vector $x^{\prime }=\sqrt{n}x/|x|$ a total of $m^{\prime }(x)$ times as a
row of the matrix $T$.

\begin{theorem}
\label{Main result}There exists a universal constant $c>0$ with the
following property. Consider any $n,N\in \mathbb{N}$, $K\geq 1$, and $%
0<\varepsilon <(2K)^{-1}$ such that%
\begin{equation*}
6\leq n\leq c(\log \varepsilon ^{-1})^{-1}\log N
\end{equation*}%
Let $(X,\left\Vert \cdot \right\Vert )$ be a real normed space of dimension $%
N$ with a basis $(e_{i})_{1}^{N}$ of basis constant $K$ that is invariant
under permutations and let $T$ be the $N\times n$ matrix defined above. Then
for all $x\in \mathbb{R}^{n}$,%
\begin{equation*}
(1-K\varepsilon )M\left( \sum_{i=1}^{n}x_{i}^{2}\right) ^{1/2}\leq
\left\Vert \sum_{i=1}^{N}\sum_{j=1}^{n}T_{i,j}x_{j}e_{i}\right\Vert \leq
(1+K\varepsilon )M\left( \sum_{i=1}^{n}x_{i}^{2}\right) ^{1/2}
\end{equation*}%
where $M=\left\Vert v\right\Vert $ and $v$ is defined by (\ref{v def}). One
can take $c=1/100$, for example. To embed $\ell _{2}^{k}$ into $X$ for $k<6$%
, simply take $n=6$ and use the matrix $\widetilde{T}$ which consists of the
first $k$ columns of $T$.
\end{theorem}

Let us briefly sketch the proof. Identify $X$ with $\mathbb{R}^{N}$ so that $%
(e_{i})_{1}^{N}$ become the standard basis vectors,%
\begin{equation*}
e_{i}=\left( 0,0\ldots ,0,1,0\ldots 0,0\right)
\end{equation*}%
Let $(\theta _{i})_{1}^{N}$ denote the rows of $T$. The empirical
distribution of the sequence $(\theta _{i})_{1}^{N}$ is the probability
measure on $\mathbb{R}^{n}$ defined by%
\begin{equation*}
\mu =\frac{1}{N}\sum_{i=1}^{N}\delta (\theta _{i})
\end{equation*}%
where $\delta (x)$ is the Dirac point mass at $x$. This measure is a
discrete approximation to normalized Haar measure on $\sqrt{n}S^{n-1}$. This
should be clear (at least in a rough sense) by (\ref{first multiplicity})
and is made precise in Lemma \ref{main approx of quantiles}. Therefore for
any $\theta \in S^{n-1}$ the measure%
\begin{equation*}
\mu _{\theta }=\text{\textrm{Proj}}_{\theta }\mu =\frac{1}{N}%
\sum_{i=1}^{N}\delta \left( \left\langle \theta ,\theta _{i}\right\rangle
\right)
\end{equation*}%
should approximate the standard normal distribution on $\mathbb{R}$. In
particular, the choice of $\theta $ has minimal effect on $\mu _{\theta }$.
Since the norm $\left\Vert \cdot \right\Vert $ is invariant under
permutations, the quantity%
\begin{equation*}
\left\Vert T\theta \right\Vert =\left\Vert \left( \left\langle \theta
,\theta _{i}\right\rangle \right) _{i=1}^{N}\right\Vert
\end{equation*}%
depends only on $\mu _{\theta }$, and does not oscillate much on $S^{n-1}$.
The result then follows by homogeneity. More details are given in Section %
\ref{mn pf}.

\section{\label{bkg}Background and preliminaries}

The symbols $B_{2}^{n}$ and $S^{n-1}$ denote the unit ball and sphere in $%
\mathbb{R}^{n}$. Let $\phi $ and $\Phi $ denote the standard normal density
and cumulative distribution,%
\begin{equation*}
\phi (t)=\frac{1}{\sqrt{2\pi }}\exp \left( -t^{2}/2\right)
\end{equation*}%
\begin{equation*}
\Phi (t)=\int_{-\infty }^{t}\phi (u)du
\end{equation*}%
Let $X=(X_{i})_{1}^{n}$ denote a random vector with the standard
multivariate normal distribution and for each $t\in \mathbb{R}$ define%
\begin{equation*}
\Phi _{n}(t)=\mathbb{P}\left\{ \frac{\sqrt{n}X_{1}}{\left(
\sum_{1}^{n}X_{i}^{2}\right) ^{1/2}}\leq t\right\}
\end{equation*}%
Note that by uniqueness of Haar measure, $\sqrt{n}X/|X|$ is uniformly
distributed on $\sqrt{n}S^{n-1}$, and for all $\theta \in S^{n-1}$,%
\begin{equation*}
\mathbb{P}\left\{ \left\langle \theta ,\sqrt{n}X/|X|\right\rangle \leq
t\right\} =\Phi _{n}(t)
\end{equation*}%
The corresponding density function, which is supported on $[-\sqrt{n},\sqrt{n%
}]$, can be written as%
\begin{equation*}
\phi _{n}(t)=\frac{(n-1)\mathrm{vol}_{n-1}(B_{2}^{n-1})}{n^{3/2}\mathrm{vol}%
_{n}(B_{2}^{n})}\left( 1-\frac{t^{2}}{n}\right) ^{(n-3)/2}
\end{equation*}%
Setting%
\begin{equation*}
\lambda _{n}=\frac{(n-1)\mathrm{vol}_{n-1}(B_{2}^{n-1})}{n^{3/2}\mathrm{vol}%
_{n}(B_{2}^{n})}
\end{equation*}%
we have the following representations%
\begin{equation*}
\lambda _{n}=\frac{(n-1)\Gamma (1+n/2)}{n^{3/2}\sqrt{\pi }\Gamma (1/2+n/2)}%
=\left( \int_{-\sqrt{n}}^{\sqrt{n}}\left( 1-\frac{t^{2}}{n}\right)
^{(n-3)/2}dt\right) ^{-1}
\end{equation*}%
The first one follows from the formula $\mathrm{vol}_{n}(B_{2}^{n})=\pi
^{n/2}/\Gamma (1+n/2)$ (see eg. \cite{Kol}) while the second follows from
the fact that $\int \phi _{n}=1$. Using the inequality $1-x\leq \exp (-x)$,
we see that $\lambda _{n}\geq 1/\sqrt{4\pi }$. By log-convexity of $\Gamma $%
, one can show that $\lambda _{n}\leq 1/\sqrt{2\pi }$. It should be clear
(for at least two different reasons) that%
\begin{equation*}
\lim_{n\rightarrow \infty }\phi _{n}(t)=\phi (t)
\end{equation*}%
\begin{equation*}
\lim_{n\rightarrow \infty }\Phi _{n}(t)=\Phi (t)
\end{equation*}%
This is an old result, see for example the discussion in \cite{DF}, and also
implies that $\lim_{n\rightarrow \infty }\lambda _{n}=1/\sqrt{2\pi }$.

\begin{lemma}
\label{normal tails}For all $t\geq \sqrt{n/(n-4)}$,%
\begin{equation*}
\frac{n}{2(n-3)t}\left( 1-\frac{t^{2}}{n}\right) \phi _{n}(t)\leq 1-\Phi
_{n}(t)\leq \frac{n}{(n-3)t}\left( 1-\frac{t^{2}}{n}\right) \phi _{n}(t)
\end{equation*}
\end{lemma}

\begin{proof}
The function $\phi _{n}$ is convex on $[(n/(n-4))^{1/2},\infty )$ and
log-concave on $\mathbb{R}$, i.e. $\log \phi _{n}(x)$ is concave. The
function is therefore sandwiched, on the interval $[t,\infty )$, between an
affine function (below) and an exponential function (above). In particular,
for all $x\geq t$,%
\begin{equation*}
\frac{-(n-3)t\phi _{n}(t)}{n(1-t^{2}/n)}(x-t)+\phi _{n}(t)\leq \phi
_{n}(x)\leq \phi _{n}(t)\exp \left( -\frac{(n-3)(x-t)t}{n(1-t^{2}/n)}\right)
\end{equation*}%
and the result follows because%
\begin{equation*}
1-\Phi _{n}(t)=\int_{t}^{\infty }\phi _{n}(u)du
\end{equation*}
\end{proof}

Taking $n\rightarrow \infty $ in Lemma \ref{normal tails} we recover the
standard estimate for $1-\Phi (t)$. The following lemma is well known in the
more general setting of log-concave functions and we include its short
proofs for completeness.

\begin{lemma}
\label{convexity}The function $\psi _{n}=\phi _{n}\circ \Phi
_{n}^{-1}:(0,1)\rightarrow (0,\infty )$ is concave.
\end{lemma}

\begin{proof}
Note that $g_{n}(t)=-\ln \phi _{n}(t)$ is convex. By the inverse function
theorem and the chain rule,%
\begin{eqnarray*}
\psi _{n}^{\prime }(t) &=&\frac{\phi _{n}^{\prime }\left( \Phi
_{n}^{-1}(t)\right) }{\phi _{n}\left( \Phi _{n}^{-1}(t)\right) } \\
&=&-g_{n}^{\prime }\left( \Phi _{n}^{-1}(t)\right)
\end{eqnarray*}%
Hence%
\begin{equation*}
\psi _{n}^{\prime \prime }(t)=\frac{-g_{n}^{\prime \prime }\left( \Phi
_{n}^{-1}(t)\right) }{\phi _{n}\left( \Phi _{n}^{-1}(t)\right) }<0
\end{equation*}
\end{proof}

\begin{lemma}
\label{Lip Phi}For all $0<a<b<1/2$,%
\begin{equation*}
\left\vert \Phi _{n}^{-1}(b)-\Phi _{n}^{-1}(a)\right\vert \leq \sqrt{\pi }%
\ln (b/a)
\end{equation*}
\end{lemma}

\begin{proof}
Consider $\psi _{n}$ as in Lemma \ref{convexity} and define $h_{n}:(0,\infty
)\rightarrow \mathbb{R}$ by $h_{n}(x)=\Phi _{n}^{-1}\left( e^{-x}\right) $.
Note that $\psi _{n}$ is positive, concave, symmetric about $x=1/2$, and%
\begin{equation*}
\lim_{x\rightarrow 0}\psi _{n}(x)=\lim_{x\rightarrow 1}\psi _{n}(x)=0
\end{equation*}%
Therefore $\psi (x)\geq 2\phi _{n}(0)x=2\lambda _{n}x\geq x/\sqrt{\pi }$ for
all $x\in (0,1/2)$. By the inverse function theorem,%
\begin{equation*}
h_{n}^{\prime }(x)=\frac{-e^{-x}}{\phi _{n}\left( \Phi
_{n}^{-1}(e^{-x})\right) }
\end{equation*}%
and $\left\vert h_{n}^{\prime }(x)\right\vert \leq \sqrt{\pi }$ for all $%
x>\ln (2)$. Hence $h_{n}$ is $\sqrt{\pi }$-Lipschitz on $(\ln (2),\infty )$
and the result follows.
\end{proof}

\begin{lemma}
\label{vol of ball}There exists a sequence $(\omega _{n})_{1}^{\infty }$
with $0<\omega _{n}<1$ and $\lim_{n\rightarrow \infty }\omega _{n}=1$ such
that for all $n\in \mathbb{N}$,%
\begin{equation*}
\mathrm{vol}_{n}(B_{2}^{n})=\left( \sqrt{\frac{2\pi e\omega _{n}}{n}}\right)
^{n}
\end{equation*}
\end{lemma}

\begin{proof}
This follows from the expression $\mathrm{vol}_{n}(B_{2}^{n})=\pi
^{n/2}\left( \Gamma (1+n/2)\right) ^{-1}$ and Sterling's approximation. See
for example Corollary 2.20 in \cite{Kol}.
\end{proof}

The following bound is well known, without dependence on $n$.

\begin{lemma}
\label{mult Gaussian decay}Let $X$ be a random vector in $\mathbb{R}^{n}$
with the standard multivariate normal distribution. Then for all $t\geq 2%
\sqrt{n}$,%
\begin{equation*}
\frac{1}{2}\leq \frac{\mathbb{P}\left\{ \left\vert X\right\vert >t\right\} }{%
n\mathrm{vol}_{n}(B_{2}^{n})(2\pi )^{-n/2}t^{n-2}\exp \left( -t^{2}/2\right) 
}\leq 4/3
\end{equation*}
\end{lemma}

\begin{proof}
By polar integration with respect to surface area measure on $S^{n-1}$ with $%
\mu _{n}(S^{n-1})=n\mathrm{vol}_{n}(B_{2}^{n})$, 
\begin{eqnarray*}
\mathbb{P}\left\{ \left\vert X\right\vert >t\right\} &=&\int_{t}^{\infty
}\int_{S^{n-1}}x^{n-1}(2\pi )^{-n/2}\exp (-x^{2}/2)d\mu _{n}(\omega )dx \\
&=&n(2\pi )^{-n/2}\mathrm{vol}_{n}(B_{2}^{n})\int_{t}^{\infty }x^{n-1}\exp
(-x^{2}/2)dx
\end{eqnarray*}%
On one hand, by the limiting case of Lemma \ref{normal tails}, 
\begin{equation*}
\int_{t}^{\infty }x^{n-1}\exp (-x^{2}/2)dx\geq t^{n-1}\int_{t}^{\infty }\exp
(-x^{2}/2)dx\geq \frac{1}{2}t^{n-2}\exp (-t^{2}/2)
\end{equation*}%
By convexity on the other hand (or basic algebra), $tx-t^{2}/2\leq x^{2}/2$
for all $x\in \mathbb{R}$. Using this together with the substitution $%
u=tx-t^{2}$, the inequality $1+z\leq \exp (z)$ valid for all $z\in \mathbb{R}
$, and the fact that (by assumption) $n/t^{2}\leq 1/4$,%
\begin{eqnarray*}
\int_{t}^{\infty }x^{n-1}\exp (-x^{2}/2)dx &\leq &\exp
(t^{2}/2)\int_{t}^{\infty }x^{n-1}\exp (-tx)dx \\
&\leq &t^{n-2}\exp \left( -t^{2}/2\right) \int_{0}^{\infty }\exp \left(
-3u/4\right) du
\end{eqnarray*}
\end{proof}

We shall also make use of the following classical result.

\begin{theorem}[Hoeffding \protect\cite{Hoef}]
Let $(\gamma _{i})_{1}^{n}$ be independent random variables with $a_{i}\leq
\gamma _{i}\leq b_{i}$. Then for all $t>0$,%
\begin{equation*}
\mathbb{P}\left\{ \left\vert \sum_{i=1}^{n}\gamma _{i}-\mathbb{E}%
\sum_{i=1}^{n}\gamma _{i}\right\vert \geq t\right\} \leq 2\exp \left( \frac{%
-2t^{2}}{\sum_{i=1}^{n}(b_{i}-a_{i})^{2}}\right)
\end{equation*}
\end{theorem}

\section{\label{mn pf}Main proof}

Consider the cumulative distribution function%
\begin{equation*}
F_{\theta }(t)=\mu _{\theta }\left( (-\infty ,t]\right) =\frac{1}{N}%
\left\vert \left\{ i\leq N:\left\langle \theta ,\theta _{i}\right\rangle
\leq t\right\} \right\vert 
\end{equation*}%
and the quantile function $F_{\theta }^{-1}:(0,1)\rightarrow \mathbb{R}$
defined by%
\begin{equation*}
F_{\theta }^{-1}(s)=\inf \left\{ t\in \mathbb{R}:F_{\theta }(t)\geq
s\right\} =\sup \left\{ t\in \mathbb{R}:F_{\theta }(t)<s\right\} 
\end{equation*}%
This is well defined even though $F_{\theta }$ is not injective and does not
have an inverse in the classical sense. Set $a=\Phi _{n}\left( 1.5\right) $
and $b=\Phi _{n}((1-17\delta )\sqrt{n})$. In Section \ref{analysis pt cloud}
we prove the following lemma.

\begin{lemma}
\label{main approx of quantiles}Consider any $\theta \in S^{n-1}$. Then,%
\begin{equation*}
\left\vert F_{\theta }^{-1}(s)-\Phi _{n}^{-1}(s)\right\vert \leq \left\{ 
\begin{array}{ccc}
20\delta \Phi _{n}{}^{-1}(s) & : & (1-b)\leq s\leq (1-a) \\ 
7\delta  & : & (1-a)<s<a \\ 
20\delta \Phi _{n}{}^{-1}(s) & : & a\leq s\leq b%
\end{array}%
\right. 
\end{equation*}%
and for $s\in (0,1-b)\cup (b,1)$, $\left\vert F_{\theta }^{-1}(s)-\sqrt{n}%
\right\vert \leq 29\delta \sqrt{n}$.
\end{lemma}

\begin{proof}[Proof of Theorem \protect\ref{Main result}]
Consider any $\theta \in S^{n-1}$. Let $(u_{i})_{1}^{N}$ be the order
statistics (non-decreasing rearrangement) of the sequence $\left(
\left\langle \theta ,\theta _{i}\right\rangle \right) _{1}^{N}$. Note that
for all $(i-1)/N<s\leq i/N$, $F_{\theta }^{-1}\left( s\right) =u_{i}$.
Define $v\in \mathbb{R}^{N}$ by%
\begin{equation}
v_{i}=\left\{ 
\begin{array}{ccc}
-\sqrt{n} & : & 0\leq i-1/2<(1-b)N \\ 
\Phi _{n}{}^{-1}((i-1/2)/N) & : & (1-b)N\leq i-1/2\leq bN \\ 
\sqrt{n} & : & bN<i-1/2\leq N%
\end{array}%
\right.   \label{v def}
\end{equation}%
By Lemma \ref{main approx of quantiles},%
\begin{equation*}
\left\vert u_{i}-v_{i}\right\vert \leq \left\{ 
\begin{array}{ccc}
29\delta \left\vert v_{i}\right\vert  & : & 0\leq i-1/2\leq (1-a)N \\ 
7\delta  & : & (1-a)N<i-1/2<aN \\ 
29\delta \left\vert v_{i}\right\vert  & : & aN\leq i-1/2\leq N%
\end{array}%
\right. 
\end{equation*}%
By the triangle inequality,%
\begin{equation*}
\left\vert (\left\Vert u\right\Vert -\left\Vert v\right\Vert )\right\vert
\leq \left\Vert u-v\right\Vert 
\end{equation*}%
Define $w,y\in \mathbb{R}^{N}$ by%
\begin{equation*}
w_{i}=\left\{ 
\begin{array}{ccc}
u_{i}-v_{i} & : & 0\leq i-1/2<(1-a)N \\ 
0 & : & (1-a)N\leq i-1/2\leq aN \\ 
u_{i}-v_{i} & : & aN<i-1/2\leq N%
\end{array}%
\right. 
\end{equation*}%
\begin{equation*}
y_{i}=\left\{ 
\begin{array}{ccc}
0 & : & 0\leq i-1/2<(1-a)N \\ 
u_{i}-v_{i} & : & (1-a)N\leq i-1/2\leq aN \\ 
0 & : & aN<i-1/2\leq N%
\end{array}%
\right. 
\end{equation*}%
It follows from this definition that $y$ has at most $(2a-1)N+1$ nonzero
coordinates. Let $y^{\prime }\in \mathbb{R}^{N}$ be defined by%
\begin{equation*}
y_{i}^{\prime }=y_{\rho (i)}
\end{equation*}%
where $\rho \in S_{N}$ is any permutation such that if $y_{i}^{\prime }=0$
and $\left\vert i-(N+1)/2\right\vert >\left\vert j-(N+1)/2\right\vert $,
then $y_{j}^{\prime }=0$ i.e. we have moved the nonzero coordinates of $y$
to the right and left hand tails of the vector (as a string of coordinates)
and placed the zero's in the middle. Then for all $i$, if $y_{i}^{\prime
}\neq 0$ then $\left\vert i-(N+1)/2\right\vert >(1-a)N/2$ so $\left\vert
v_{i}\right\vert \geq \Phi _{n}{}^{-1}\left( 0.5+(1-a)/2\right) >5\times
10^{-3}$ and%
\begin{equation*}
\left\vert y_{i}^{\prime }\right\vert \leq 1400\delta \left\vert
v_{i}\right\vert 
\end{equation*}%
Here we have used the inequality $1-x\geq \exp (-2e^{2}x/(e^{2}-1))$ valid
for all $0\leq x\leq 1-e^{-2}$, which implies that%
\begin{equation*}
a=1-\lambda _{n}\int_{1.5}^{\sqrt{n}}\left( 1-\frac{t^{2}}{n}\right)
^{(n-3)/2}dt\leq 1-\frac{1}{\sqrt{4\pi }}\int_{1.5}^{2e^{-1}\sqrt{e^{2}-1}%
}\exp \left( -\frac{e^{2}t^{2}}{(e^{2}-1)}\right) dt\leq 0.996
\end{equation*}%
and%
\begin{equation*}
\Phi _{n}{}^{-1}\left( 0.5+(1-a)/2\right) \geq \Phi _{n}{}^{-1}\left(
0.502\right) \geq 2\times 10^{-3}/\lambda _{n}
\end{equation*}%
This implies that%
\begin{eqnarray*}
\left\Vert u-v\right\Vert  &=&\left\Vert w+y\right\Vert  \\
&\leq &\left\Vert w\right\Vert +\left\Vert y\right\Vert  \\
&\leq &29\delta K\left\Vert v\right\Vert +1400\delta K\left\Vert
v\right\Vert 
\end{eqnarray*}%
The result now follows with $M=\left\Vert v\right\Vert $.
\end{proof}

\section{\label{analysis pt cloud}Analysis of the point cloud}

Let $X=(X_{i})_{1}^{n}$ be a random vector in $\mathbb{R}^{n}$ with the
standard multivariate normal distribution (i.e. mean zero and identity
covariance matrix). Let $[t]$ denote the closest integer function of $t\in 
\mathbb{R}$, where we round closer to zero if $t$ is midway between two
integers. Let $Y=(Y_{i})_{1}^{n}$ where $Y_{i}=[\sigma X_{i}]$, and let $%
Z_{i}=Y_{i}-\sigma X_{i}$. Consider the random vector $Q$ defined by%
\begin{equation*}
Q=\sqrt{n}|Y|^{-1}Y\cdot 1_{\{Y\neq 0\}\cup \{\left\vert Y\right\vert \leq
\alpha \sqrt{n}\}}
\end{equation*}%
where $1_{\{Y\neq 0\}\cup \{\left\vert Y\right\vert \leq \alpha \sqrt{n}\}}$
is the indicator function of the event $\left\{ Y\neq 0\right\} \cup \left\{
\left\vert Y\right\vert \leq \alpha \sqrt{n}\right\} $. Let $\theta \in
S^{n-1}$ denote an arbitrary unit vector.

\begin{lemma}
\label{first estimate on Y}For all $1\leq t\leq (1-3\delta )\sqrt{n}$,%
\begin{equation}
(1-\delta )(1-\Phi _{n}((1+3\delta )t)\leq \mathbb{P}\left\{
\sum_{i=1}^{n}\theta _{i}Q_{i}>t\right\} \leq (1+\delta )(1-\Phi
_{n}((1-2\delta )t))  \label{distribution of Y}
\end{equation}
\end{lemma}

\begin{proof}
By the union bound,%
\begin{equation*}
\mathbb{P}\left\{ \sum_{i=1}^{n}\theta _{i}Q_{i}\leq t\right\} \leq \mathbb{P%
}\left\{ \sqrt{n}|Y|^{-1}\sum_{i=1}^{n}\theta _{i}Y_{i}\leq t\right\} +%
\mathbb{P}\left\{ Y=0\right\} +\mathbb{P}\left\{ |Y|>\alpha \sqrt{n}\right\} 
\end{equation*}%
By the triangle inequality $\left\vert |\sigma X|-|Y|\right\vert \leq
\left\vert Z\right\vert \leq \sqrt{n}/2$ and by the definitions $\alpha
=2\delta ^{-4}(\log \delta ^{-1})^{1/2}$ and $\sigma =\delta ^{-4}$, we have 
$(\alpha -1/2)/\sigma \geq 2$. By Lemma \ref{mult Gaussian decay} which
bounds $|X|$ and Lemma \ref{vol of ball}\ on $\mathrm{vol}_{n}(B_{2}^{n})$,
we have 
\begin{eqnarray*}
&&\mathbb{P}\left\{ |Y|>\alpha \sqrt{n}\right\}  \\
&\leq &\mathbb{P}\{|X|>\sigma ^{-1}(\alpha -1/2)\sqrt{n}\} \\
&\leq &2n\mathrm{vol}_{n}(B_{2}^{n})(2\pi )^{-n/2}(\sigma ^{-1}(\alpha -1/2)%
\sqrt{n})^{n-2}\exp \left( -1/2(\sigma ^{-1}(\alpha -1/2)\sqrt{n}%
)^{2}\right)  \\
&\leq &2e^{n/2}\sigma ^{-n+2}(\alpha -1/2)^{n-2}\exp (-1/2\sigma
^{-2}(\alpha -1/2)^{2}n) \\
&\leq &2^{n-1}e^{n/2}\left( \log \delta ^{-1}\right) ^{(n-2)/2}\delta
^{19n/10} \\
&\leq &e^{n/2}\delta ^{n}
\end{eqnarray*}%
Since $\left\Vert \phi \right\Vert _{\infty }\leq (2\pi )^{-1/2}$, $\mathbb{P%
}\left\{ Y=0\right\} \leq (2\pi )^{-n/2}\delta ^{4n}$. Using the union bound
again, as well as the equations $Y=\sigma X+Z$ and $\sigma =\delta ^{-4}$, 
\begin{eqnarray*}
&&\mathbb{P}\left\{ \frac{\sqrt{n}}{|Y|}\sum_{i=1}^{n}\theta _{i}Y_{i}\leq
t\right\}  \\
&\leq &\mathbb{P}\left\{ \frac{\left\vert \sigma X\right\vert }{\left\vert
Y\right\vert }\frac{\sqrt{n}}{|X|}\sum_{i=1}^{n}\theta _{i}X_{i}\leq
(1+\delta )t\right\} +\mathbb{P}\left\{ \frac{\left\vert \sigma X\right\vert 
}{\left\vert Y\right\vert }\frac{\sqrt{n}}{|\sigma X|}\left\vert
\sum_{i=1}^{n}\theta _{i}Z_{i}\right\vert >\delta t\right\}  \\
&\leq &\mathbb{P}\left\{ \frac{\sqrt{n}}{|X|}\sum_{i=1}^{n}\theta
_{i}X_{i}\leq (1+\delta )^{2}t\right\} +\mathbb{P}\left\{ \frac{\left\vert
\sigma X\right\vert }{\left\vert Y\right\vert }<(1+\delta )^{-1}\right\}  \\
&&+\mathbb{P}\left\{ \left\vert \sum_{i=1}^{n}\theta _{i}Z_{i}\right\vert
>\delta ^{-2}t/2\right\} +\mathbb{P}\left\{ \frac{\left\vert \sigma
X\right\vert }{\left\vert Y\right\vert }>(1-\delta )^{-1}\right\} +\mathbb{P}%
\left\{ \frac{\sqrt{n}}{|X|}>\delta ^{-1}\right\} 
\end{eqnarray*}%
As noted in Section \ref{bkg},%
\begin{equation*}
\mathbb{P}\left\{ \frac{\sqrt{n}}{|X|}\sum_{i=1}^{n}\theta _{i}X_{i}\leq
(1+\delta )^{2}t\right\} =\Phi _{n}\left( (1+\delta )^{2}t\right) 
\end{equation*}%
By Hoeffding's inequality,%
\begin{equation*}
\mathbb{P}\left\{ \left\vert \sum_{i=1}^{n}\theta _{i}Z_{i}\right\vert \geq
\delta ^{-2}t/2\right\} \leq 2\exp (-\delta ^{-4}t^{2}/2)
\end{equation*}%
and $\mathbb{P}\left\{ |Z|\geq \delta ^{-1}\sqrt{n}\right\} \leq 2\exp
(-\delta ^{-4}n)$. By Lemma \ref{vol of ball} on $\mathrm{vol}_{n}\left(
B_{2}^{n}\right) $, 
\begin{equation*}
\mathbb{P}\left\{ |X|\leq \delta \sqrt{n}\right\} \leq (2\pi )^{-n/2}\mathrm{%
vol}_{n}\left( \delta \sqrt{n}B_{2}^{n}\right) \leq e^{n/2}\delta ^{n}
\end{equation*}%
This implies, via the inequality $\left\vert \left\vert \sigma X\right\vert
-\left\vert Y\right\vert \right\vert \leq \left\vert Z\right\vert $, that%
\begin{equation*}
\mathbb{P}\left\{ (1-\delta )\left\vert \sigma X\right\vert \leq \left\vert
Y\right\vert \leq (1+\delta )\left\vert \sigma X\right\vert \right\} \geq
1-2e^{n/2}\delta ^{n}
\end{equation*}%
It follows that%
\begin{equation*}
\mathbb{P}\left\{ \sqrt{n}|Y|^{-1}\sum_{i=1}^{n}\theta _{i}Q_{i}\leq
t\right\} \leq \Phi _{n}((1+\delta )^{2}t)+5e^{n/2}\delta ^{n}+2\exp
(-\delta ^{-4}t^{2}/2)
\end{equation*}%
By Lemma \ref{normal tails}%
\begin{equation*}
\delta \left( 1-\Phi _{n}((1+\delta )^{2}t)\right) \geq \delta \left( 1-\Phi
_{n}((1-\delta )\sqrt{n}\right) )\geq \frac{1}{5\sqrt{2\pi n}}\delta
^{(n+1)/2}\geq 10e^{n/2}\delta ^{n}
\end{equation*}%
Our remaining task is to show that $4\exp (-\delta ^{-4}t^{2}/2)\leq \delta
\left( 1-\Phi _{n}((1+\delta )^{2}t)\right) $. Using Lemma \ref{normal tails}
again, it suffices to have%
\begin{equation}
4\exp (-\delta ^{-4}t^{2}/2)\leq \frac{\delta }{6t\sqrt{\pi }}\left( 1-\frac{%
(1+\delta )^{4}t^{2}}{n}\right) ^{(n-1)/2}  \label{12wdr}
\end{equation}%
We now consider two cases. In Case 1, $t\geq \delta \sqrt{n}$. Since $%
(1+\delta )^{4}t^{2}\leq (1-\delta )n$, (\ref{12wdr}) is implied by%
\begin{equation*}
\frac{1}{2}\delta ^{-4}t^{2}\geq \log (24\sqrt{\pi })+\log t+\frac{n+1}{2}%
\log \delta ^{-1}
\end{equation*}%
which holds by the assumption of Case 1. In Case 2, $t<\delta \sqrt{n}$.
Using the fact that $\log (1-x)^{-1}\leq 2x$ whenever $0<x<1/2$, a
sufficient condition for (\ref{12wdr}) to hold is that%
\begin{equation*}
\delta ^{-4}t^{2}\geq 2\log (24\sqrt{\pi })+2\log t+2\log \delta ^{-1}+3t^{2}
\end{equation*}%
which is true by the bounds imposed on $t$ and $\delta $. This proves the
left hand inequality in (\ref{distribution of Y}). The right hand inequality
follows similar lines.
\end{proof}

\begin{lemma}
\label{middle bulk}For all $0\leq t\leq 2$,%
\begin{equation*}
\Phi _{n}(t-5\delta )\leq \mathbb{P}\left\{ \sum_{i=1}^{n}\theta
_{i}Q_{i}\leq t\right\} \leq \Phi _{n}(t+5\delta )
\end{equation*}
\end{lemma}

\begin{proof}
We refer the reader to the proof of Lemma \ref{first estimate on Y} for many
of the details. By the union bound, a lower bound on the growth rate of $%
\Phi _{n}$ (via the mean value theorem),%
\begin{eqnarray*}
&&\mathbb{P}\left\{ \sum_{i=1}^{n}\theta _{i}Q_{i}\leq t\right\}  \\
&\leq &\mathbb{P}\left\{ \sqrt{n}|Y|^{-1}\sum_{i=1}^{n}\theta _{i}Y_{i}\leq
t\right\} +\mathbb{P}\left\{ Y=0\right\} +\mathbb{P}\left\{ |Y|>\alpha \sqrt{%
n}\right\}  \\
&\leq &\mathbb{P}\left\{ \frac{\left\vert \sigma X\right\vert }{|Y|}\frac{%
\sqrt{n}}{|X|}\sum_{i=1}^{n}\theta _{i}X_{i}\leq t+\delta ^{2}\right\} +%
\mathbb{P}\left\{ \frac{\left\vert \sigma X\right\vert }{|Y|}\frac{\sqrt{n}}{%
|\sigma X|}\left\vert \sum_{i=1}^{n}\theta _{i}Z_{i}\right\vert >\delta
^{2}\right\} +2e^{n/2}\delta ^{n} \\
&\leq &\mathbb{P}\left\{ \frac{\sqrt{n}}{|X|}\sum_{i=1}^{n}\theta
_{i}X_{i}\leq (1+\delta )(t+\delta ^{2})\right\} +\mathbb{P}\left\{ \frac{%
\left\vert \sigma X\right\vert }{|Y|}<(1+\delta )^{-1}\right\}
+2e^{n/2}\delta ^{n} \\
&&+\mathbb{P}\left\{ \left\vert \sum_{i=1}^{n}\theta _{i}Z_{i}\right\vert
>(1-\delta )\delta ^{-1}\right\} +\mathbb{P}\left\{ \frac{\left\vert \sigma
X\right\vert }{|Y|}>(1-\delta )^{-1}\right\} +\mathbb{P}\left\{ \sqrt{n}%
|X|^{-1}>\delta ^{-1}\right\}  \\
&\leq &\Phi _{n}((1+\delta )(t+\delta ^{2}))+6e^{n/2}\delta ^{n}+2\exp
(-\delta ^{-2})\leq \Phi _{n}(t+5\delta )
\end{eqnarray*}%
The lower bound follows by similar reasoning, and using an upper bound on
the growth rate of $\Phi _{n}$.
\end{proof}

\begin{lemma}
\label{some conc F}For all $0\leq t\leq 2$,%
\begin{equation*}
\Phi _{n}(t-6\delta )\leq F_{\theta }(t)\leq \Phi _{n}(t+6\delta )
\end{equation*}%
and for all $1\leq t\leq (1-3\delta )\sqrt{n}$,%
\begin{equation*}
(1-4\delta )(1-\Phi _{n}((1+3\delta )t))\leq 1-F_{\theta }(t)\leq (1+4\delta
)(1-\Phi _{n}((1-2\delta )t))
\end{equation*}
\end{lemma}

\begin{proof}
Consider the distribution of $Y\cdot 1_{\left\{ |Y|\leq \alpha \sqrt{n}%
\right\} }$ and the empirical distribution of the sequence $(x_{i})_{1}^{N}$%
, where each integer point $x$ is repeated $m^{\prime }(x)$ times in the
sequence,%
\begin{equation*}
\eta =\frac{1}{N}\sum_{i=1}^{N}\delta (x_{i})
\end{equation*}%
These two probability measures are both discrete measures supported on $%
\mathbb{Z}^{n}\cap \left( \alpha \sqrt{n}B_{2}^{n}\right) $ that we now
compare. For each nonzero $x\in \mathbb{Z}^{n}\cap \left( \alpha \sqrt{n}%
B_{2}^{n}\right) $ we have, by definition of $m^{\prime }(\cdot )$ and $Y$, 
\begin{equation*}
\left\vert N\cdot \mathbb{P}\left\{ Y=x\right\} -m^{\prime }(x)\right\vert
\leq 1
\end{equation*}%
as well as,%
\begin{eqnarray*}
\mathbb{P}\left\{ Y=x\right\}  &\geq &(2\pi )^{-n/2}\sigma ^{-n}\exp \left(
-1/2\sigma ^{-2}(\alpha +1/2)^{2}n\right)  \\
&=&\exp \left[ -\frac{1}{2}\left( \sigma ^{-2}(\alpha +1/2)^{2}+\log 2\pi
+\log \sigma ^{2}\right) n\right] 
\end{eqnarray*}%
By the bounds imposed on $N$,%
\begin{equation}
N\geq \delta ^{-1}\sigma \exp \left[ \frac{1}{2}\left( \sigma ^{-2}(\alpha
+1/2)^{2}+\log 2\pi +\log \sigma ^{2}\right) n\right]   \label{by bound on N}
\end{equation}%
This implies that%
\begin{equation*}
\left( 1-\delta \sigma ^{-1}\right) \mathbb{P}\left\{ Y=x\right\} \leq \eta
(\{x\})\leq \left( 1+\delta \sigma ^{-1}\right) \mathbb{P}\left\{
Y=x\right\} 
\end{equation*}%
Consequently, for any Borel set $E\subset \mathbb{R}^{n}$ with $0\notin E$,%
\begin{equation*}
\left( 1-\delta \sigma ^{-1}\right) \mathbb{P}\left\{ Y\in E\cap \alpha 
\sqrt{n}B_{2}^{n}\right\} \leq \eta (E)\leq \left( 1+\delta \sigma
^{-1}\right) \mathbb{P}\left\{ Y\in E\cap \alpha \sqrt{n}B_{2}^{n}\right\} 
\end{equation*}%
Hence, the same is true when we take the radial projections of these
measures onto $\sqrt{n}S^{n-1}$, i.e.%
\begin{equation*}
\left( 1-\delta \sigma ^{-1}\right) \mathbb{P}\left\{ Q\in E\right\} \leq
\mu (E)\leq \left( 1+\delta \sigma ^{-1}\right) \mathbb{P}\left\{ Q\in
E\right\} 
\end{equation*}%
Lastly,%
\begin{equation*}
\left\vert \mu (\{0\})-\mathbb{P}\left\{ Q=0\right\} \right\vert =\left\vert
\mu (\mathbb{R}^{n}\backslash \{0\})-\mathbb{P}\left\{ Q\neq 0\right\}
\right\vert \leq \delta \sigma ^{-1}\mathbb{P}\left\{ Q\neq 0\right\} \leq
\delta \sigma ^{-1}
\end{equation*}%
The result now follows from Lemma \ref{middle bulk} (using properties of $%
\Phi _{n}$ as before) and Lemma \ref{first estimate on Y}.
\end{proof}

\begin{lemma}
\label{tail mult conc Y}For all $1\leq t\leq (1-3\delta )\sqrt{n}$%
\begin{equation*}
\Phi _{n}((1-10\delta )t)\leq F_{\theta }(t)\leq \Phi _{n}((1+12\delta )t)
\end{equation*}
\end{lemma}

\begin{proof}
By Lemma \ref{Lip Phi} and Lemma \ref{some conc F}, as well as the equation $%
\Phi _{n}(-x)=1-\Phi _{n}(x)$ and the fact that $t\geq 1$,%
\begin{eqnarray*}
1-F_{\theta }(t) &\leq &\Phi _{n}\circ \Phi _{n}{}^{-1}\left[ (1+4\delta
)(1-\Phi _{n}((1-2\delta )t))\right]  \\
&\leq &\Phi _{n}\left[ \sqrt{\pi }\ln (1+4\delta )-t(1-2\delta )\right]  \\
&\leq &1-\Phi _{n}(t(1-2\delta )-11\delta t)
\end{eqnarray*}%
The other side of the inequality follows similarly, using the fact that $\ln
(1-4\delta )\geq -5\delta $.
\end{proof}

\begin{proof}[Proof of Lemma \protect\ref{main approx of quantiles}]
Without loss of generality, we may assume that $s\geq 1/2$. Consider 3
cases. In Case 1,%
\begin{equation*}
\Phi _{n}((1-17\delta )\sqrt{n})<s<1
\end{equation*}%
In this case, using Lemma \ref{tail mult conc Y},%
\begin{equation*}
F_{\theta }\left( (1-29\delta )\sqrt{n}\right) \leq \Phi _{n}((1+12\delta
)(1-29\delta )\sqrt{n})<s
\end{equation*}%
which implies that $F_{\theta }^{-1}(s)\geq (1-29\delta )\sqrt{n}$. On the
other hand, since $\mu $ is supported on $\sqrt{n}S^{n-1}$, $F_{\theta
}^{-1}(s)\leq \sqrt{n}$. In Case 2,%
\begin{equation*}
\Phi _{n}\left( 1.5\right) \leq s\leq \Phi _{n}((1-17\delta )\sqrt{n})
\end{equation*}%
and it follows by Lemma \ref{tail mult conc Y} and the fact that $\Phi _{n}$
is strictly increasing on $[-\sqrt{n},\sqrt{n}]$ that%
\begin{equation*}
F_{\theta }\left( (1+16\delta )^{-1}\Phi _{n}{}^{-1}\left( s\right) \right)
<s\leq F_{\theta }\left( (1-14\delta )^{-1}\Phi _{n}{}^{-1}\left( s\right)
\right) 
\end{equation*}%
which implies $(1-16\delta )\Phi _{n}{}^{-1}\left( s\right) \leq F_{\theta
}^{-1}(s)\leq (1+20\delta )\Phi _{n}{}^{-1}\left( s\right) $. In Case 3, $%
1/2\leq s<\Phi _{n}(1.5)$ and it follows by Lemma \ref{some conc F} that%
\begin{equation*}
F_{\theta }\left( \Phi _{n}{}^{-1}\left( s\right) -7\delta \right)
<s<F_{\theta }\left( \Phi _{n}{}^{-1}\left( s\right) +7\delta \right) 
\end{equation*}%
and that $\Phi _{n}{}^{-1}\left( s\right) -7\delta <F_{\theta }^{-1}(s)<\Phi
_{n}{}^{-1}\left( s\right) +7\delta $.
\end{proof}


\begin{thebibliography}{99}
\bibitem{AM06} Artstein-Avidan, S., Milman, V.D.: Logarithmic reduction of
the level of randomness in some probabilistic geometric constructions. J.
Funct. Anal. 235 (1), 297-329 (2006)

\bibitem{AM07} Artstein-Avidan, S., Milman, V.D.: Using Rademacher
permutations to reduce randomness. Algebra i Analiz 19 (1), 23-45 (2007).
Translated in: St. Petersburg Math. J. 19 (1), 15-31 (2008)

\bibitem{BoLiMi} Bourgain, J., Lindenstrauss, J.: Almost Euclidean sections
in spaces with a symmetric basis. Geometric aspects of functional analysis
(1987-88), 278-288, Lecture Notes in Math., 1376, Springer, Berlin (1989)

\bibitem{DF} Diaconis, P., Freedman, D.: A dozen de Fenetti-style results in
search of a theory. Ann. Inst. H. Poincar\'{e} Probab. Statist. 23 (2),
397-423 (1987)

\bibitem{GLR} Guruswami, V., Lee, J., Razborov, A.: Almost Euclidean
subspaces of $\ell _{1}^{N}$ via expander codes. Proceedings of the
Nineteenth Annual ACM-SIAM Symposium on Discrete Algorithms, 353-362, ACM,
New York (2008)

\bibitem{GLW08} Guruswami, V., Lee, J., Wigderson, A.: Euclidean sections of 
$\ell _{1}^{N}$ with sublinear randomness and error-correction over the
reals. Approximation, randomization and combinatorial optimization, 444-454,
Lecture Notes in Comput. Sci., 5171, Springer, Berlin, (2008)

\bibitem{Hoef} Hoeffding, W.: Probability inequalities for sums of bounded
random variables. J. Amer. Statist. Assoc. 58, 13-30 (1963)

\bibitem{Ind00} Indyk, P.: Dimensionality reduction techniques for proximity
problems. Proceedings of the eleventh annual ACM-SIAM Symposium on Discrete
Algorithms (San Francisco, CA, 2000), 371-378, ACM, New York (2000)

\bibitem{Ind07} Indyk, P.: Uncertainty principles, extractors, and explicit
embedding of $\ell _{2}$ into $\ell _{1}$. STOC'07-Proceedings of the 39th
Annual ACM Symposium on Theory of Computing, 615-620, ACM, New York (2007)

\bibitem{IndSz} Indyk, P., Szarek, S.: Almost-Euclidean subspaces of $\ell
_{1}^{N}$ via tensor products: a simple approach to randomness reduction.
Approximation, randomization, and combinatorial optimization, 632-641,
Lecture Notes in Comput. Sci., 6302, Springer, Berlin, 2010.

\bibitem{JoSc} Johnson, W. B., Schechtman, G.: Finite dimensional subspaces
of $L_{p}$. Handbook of the geometry of Banach spaces, Vol. I, 837-870,
North-Holland, Amsterdam (2001)

\bibitem{KaSz} Kashin, B., Szarek, S.: The Knaster problem and the geometry
of high-dimensional cubes. C. R. Math. Acad. Sci. Paris 336 (11), 931-936
(2003)

\bibitem{KoKo} Koldobsky, A., K\"{o}nig, H.: Aspects of the isometric theory
of Banach spaces. Handbook of the geometry of Banach spaces, Vol. I,
899-939, North-Holland, Amsterdam (2001)

\bibitem{Kol} Koldobsky, A.: Fourier analysis in convex geometry.
Mathematical Surveys and Monographs, 116. American Mathematical Society,
Providence, RI (2005)

\bibitem{LS} Lovett, S., Sodin, S.: Almost Euclidean sections of the $N$%
-dimensional cross-polytope using $O(N)$ random bits. Commun. Contemp. Math.
10 (4), 477-489 (2008)

\bibitem{Mil} Milman, V. D.: A new proof of the theorem of A. Dvoretzky on
sections of convex bodies. Funct. Anal. Appl. 5, 28-37 (1971)

\bibitem{MiSc} Milman, V. D., Schechtman, G.: Asymptotic theory of
finite-dimensional normed spaces. Lecture Notes in Mathematics, 1200.
Springer-Verlag, Berlin (1986)

\bibitem{Mil88} Milman, V. D.: A few observations on the connections between
local theory and some other fields. Geometric aspects of functional analysis
(1986/87), 283-289, Lecture Notes in Math., 1317, Springer, Berlin (1988)

\bibitem{Sch0} Schechtman, G.: Special orthogonal splittings of $L_{1}^{2k}$%
. Israel J. Math. 139, 337-347 (2004)

\bibitem{Sch1} Schechtman, G.: Two observations regarding embedding subsets
of Euclidean space in normed spaces. Adv. Math. 200 (1), 125-135 (2006)

\bibitem{Sch2} Schechtman, G.: Euclidean sections of convex bodies.
Available online at www.wisdom.weizmann.ac.il/\symbol{126}gideon/

\bibitem{SzICM} Szarek, S.: Convexity, complexity and high dimensions.
Proceedings of the International Congress of Mathematicians (Madrid 2006),
vol. II, 1599-1621. European Math. Soc. (2006) Full document available at:
www.icm2006.org. Slides available at www.cwru.edu/artsci/math/szarek/

\bibitem{Tich} Tikhomirov, K.: Almost Euclidean sections in symmetric spaces
and concentration of order statistics. J. Funct. Anal. 265 (9), 2074-2088,
(2013)
\end{thebibliography}
\end{document}